\documentclass[11pt]{article}

\usepackage{amsmath,amssymb,amsthm,setspace,graphicx,mathtools}
\usepackage[text={6.5in,8.4in}, top=1.3in]{geometry}

\newtheorem{theorem}{Theorem}[section]

\newtheorem{lemma}[theorem]{Lemma}

\newtheorem{problem}[theorem]{Problem}
\newtheorem{definition}[theorem]{Definition}

\newcommand{\F}{{\mathcal F}}
\newcommand{\T}{{\mathcal T}}

\begin{document}
	
\title{
		Perfect and nearly perfect separation dimension of complete and random graphs
		\thanks{This research was supported by the Israel Science Foundation (grant No. 1082/16).}
	}
	
\author{
		Raphael Yuster
		\thanks{Department of Mathematics, University of Haifa, Haifa 3498838, Israel. Email: raphael.yuster@gmail.com}}
	
\date{}
	
\maketitle
	
\setcounter{page}{1}
	
\begin{abstract}
	
The separation dimension of a hypergraph $G$ is the smallest natural number $d$ for which there is an embedding of $G$ into $\mathbb{R}^d$, such that any pair of disjoint 
edges is separated by some hyperplane normal to one of the axes. The perfect separation dimension further requires that any pair of disjoint edges is separated by the same amount of such (pairwise nonparallel) hyperplanes.
While it is known that for any fixed $r \ge 2$, the separation dimension of any $n$-vertex $r$-graph
is $O(\log n)$, the perfect separation dimension is much larger.
In fact, no polynomial upper-bound for the perfect separation dimension of $r$-uniform hypergraphs is known.

In our first result we essentially resolve the case $r=2$, i.e. graphs.
We prove that the perfect separation dimension of $K_n$ is linear in $n$,
up to a small polylogarithmic factor. In fact, we prove it is at least $n/2-1$ and at most $n(\log n)^{1+o(1)}$.

Our second result proves that the perfect separation dimension of almost all graphs is also linear in $n$,
up to a logarithmic factor. This follows as a special case of a more general result showing that
the perfect separation dimension of the random graph $G(n,p)$ is w.h.p. $\Omega(n p /\log n)$ for a wide range of values of $p$, including all constant $p$.

Finally, we prove that significantly relaxing perfection to just requiring that any pair of disjoint edges
of $K_n$ is separated the same number of times up to a difference of  $c \log n$ for some absolute constant $c$, still requires the dimension to be $\Omega(n)$. This is perhaps surprising as it is known that if we allow a difference of $7\log_2 n$, then the dimension reduces to $O(\log n)$.

\vspace*{3mm}

\end{abstract}

\section{Introduction}

The separation dimension of a hypergraph $G$ is the smallest natural number $d$ for which there is an embedding of $G$ into $\mathbb{R}^d$, such that any pair of disjoint edges is separated by some hyperplane normal to one of the axes. Equivalently, it is the smallest possible size of a {\em separating family} (also called {\em pairwise-suitable family} in \cite{BCG+-2014}) of total orders of $V(G)$, where a family $\F$ of total orders of $V(G)$ is separating if
for any two disjoint edges of $G$, there exists at least one element of $\F$ in which all the vertices
of one edge precede those of the other. Let $\pi(G)$ denote the separation dimension of $G$.

Separation dimension of graphs and hypergraphs has been studied by several researchers. Basavaraju  et al. \cite{BCG+-2014} were the first to define and study it,
motivated by its interesting connection with {\em boxicity}, a well-studied geometric representation of graphs.
In particular, they have shown that for every fixed $r \ge 2$ it holds that $\pi(K_n^r)=\Theta(\log n)$\footnote{Hereafter, when necessary, assume that $\log n=\log_2 n$.} where $K_n^r$ is the complete $r$-graph with $n$ vertices.
Alon et al. \cite{ABC+-2015} proved that if $G$ is a graph with maximum degree at most $d$, then $\pi(G)$ is upper bounded by
an almost linear function of $d$ and provided constructions where $\pi(G) \ge d/2$.
Scott and Wood \cite{SW-2018} improved the aforementioned result on bounded degree graphs showing that
$\pi(G) \le 20 d$.
Recently, Alon et al. \cite{ABC+-2018} proved that the separation dimension of a $d$-degenerate graph on $n$ vertices is $O(d \log \log n)$ and that the $\log \log n$ term is required already for $2$-degenerate graphs.
They have also shown that graphs with bounded separation dimension cannot be very dense.
Additional papers on separation dimension and some of its variants can be found in \cite{BDL-2017,CMS-2011,LW-2018,ZRM+-2018}.

Separation dimension is also closely related to the well-studied problem of minimum 
$t$-sequence covering arrays (a.k.a. totally scrambling sets of permutations), where a $t$-sequence covering array is a set of total orders of $[n]$ such that every $t$-sequence of distinct elements of $[n]$ is a subsequence of at least one of the orders. The study of $t$-sequence covering arrays was initiated by Spencer \cite{spencer-1972} (generalizing a problem of Dushnik \cite{dushnik-1950}) and was subsequently studied in many papers (see e.g. \cite{CCHZ-2013,furedi-1996,radhakrishnan-2003,tarui-2008} and their references).
It is immediate to see that any upper bound for the size of a $2r$-sequence covering array implies the same upper bound for $\pi(K_n^r)$. The design-theoretic counterpart of $t$-sequence covering arrays are also well studied;
these are $t-(n,n,\lambda)$ {\em directed designs} (see \cite{CD-2006} Chapter 20 as well as \cite{CCHZ-2013}
for known results on the existence and non-existence of $t-(n,n,\lambda)$ and, more generally, $t-(n,k,\lambda)$ directed designs). Recalling standard design-theoretic notation, a $t-(n,n,\lambda)$ directed design requires
that every $t$-sequence of distinct elements of $[n]$ is a subsequence of precisely $\lambda$ of the total orders. When considering $t-(n,n,\lambda)$ directed designs ($t$ fixed), the natural goal is to
determine the minimum $\lambda$ for their existence. Unfortunately, we do not yet know much about that minimum $\lambda$. In fact, for all $t \ge 4$, no polynomial lower bound for $\lambda$ is known \cite{yuster-2020}.
As we shall see below, this particular difficulty (already at $t=4$) is one of the most intriguing reasons to
look at the design-theoretic counterpart of separation dimension (which we call perfect separation dimension below) in addition to it being a very natural problem.

For a separating family $\F$ (hereafter we allow $\F$ to be a multiset) and for a pair $\{e,f\}$ of disjoint edges of $G$, let $c_{\F}(\{e,f\})$ be the number of total orders in $\F$ that separate $e$ and $f$.
\begin{definition}[perfect separation dimension]
	A separating family $\F$ for a hypergraph $G$ is {\em perfect} if there exists a positive integer $\lambda$ such that
	for any pair $\{e,f\}$ of disjoint edges of $G$ it holds that $c_{\F}(\{e,f\})=\lambda$.
	We call $\lambda$ the {\em multiplicity} of $\F$.
	Let ${\rm PSD}(G)$ be the {\em perfect separation dimension of $G$}, namely the smallest possible size of a
	perfect separating family.
\end{definition}

Notice that for any uniform hypergraph $G$, the parameter ${\rm PSD}(G)$ is well-defined, as the family of all total orders of $V(G)$ is a perfect separating family. Hence, ${\rm PSD}(G) \le n!$ is a trivial upper bound for any $r$-graph with $n$ vertices. Another easily observed fact is that if $G'$ is a sub(hyper)graph of $G$, then ${\rm PSD}(G') \le {\rm PSD}(G)$. In particular, if $G$ is an $n$-vertex $r$-graph, then ${\rm PSD}(G) \le {\rm PSD}(K_n^r)$.
As mentioned earlier, $\pi(K_n^r) = \Theta(\log n)$ \cite{BCG+-2014}.
On the other hand, determining the order of magnitude of ${\rm PSD}(K_n^r)$ seems more involved.
In fact, the following very crude open problem emerges:
\begin{problem}\label{prob:1}
	Let $r \ge 2$ be fixed. Is ${\rm PSD}(K_n^r)$ upper-bounded by a polynomial function of $n$ ?
\end{problem}
Our first result resolves, in a strong sense, the first interesting case of Problem \ref{prob:1},
namely the case $r=2$ of graphs. It determines ${\rm PSD}(K_n)$ up to a small polylogarithmic factor.

\begin{theorem}\label{t:main}
	For all $n \ge 4$ it holds that
	$$
	n/2-1 \le {\rm PSD}(K_n) \le n(\log n)^{1+o(1)}\,.
	$$
\end{theorem}
\noindent We remain with the following more delicate open problem for graphs.
\begin{problem}
	Is ${\rm PSD}(K_n)$ linear in $n$?
\end{problem}
Let us briefly explain the major difficulty in proving the upper bound in Theorem \ref{t:main}.
Recall the immediate observation that any upper bound for the minimum $\lambda$ for which a $2r-(n,n,\lambda)$ directed design exists, serves as an upper bound for $PSD(K_n^r)$.
But, as mentioned earlier, even for the smallest case of $r=2$, namely graphs, we do not know whether there exists a $4-(n,n,\lambda)$ directed design in which $\lambda$ is polynomial in $n$. So, we cannot use
$4-(n,n,\lambda)$ directed designs to solve Problem \ref{prob:1} for graphs. The situation is even worse since it
{\em is known} that in any $4-(n,n,\lambda)$ directed design, $\lambda$ is {\em at least} quadratic in $n$
\cite{yuster-2020}, so we will {\em never} be able to use $4-(n,n,\lambda)$ directed designs to obtain Theorem \ref{t:main}. On the other hand, $4-(n,n,\lambda)$ directed designs and $PSD(K_n)$ seem ``so close''.
Indeed, in the former, we need to satisfy $n(n-1)(n-2)(n-3)=\Theta(n^4)$ requirements as these are the number of sequences of four distinct elements of $n$ while in the latter, we need to satisfy $n(n-1)(n-2)(n-3)/8=\Theta(n^4)$ requirements as these are the number of disjoint pairs of edges of $K_n$.
So, one may stipulate that the two parameters are close, while apparently, by Theorem \ref{t:main}, they are so different. The upper bound proof of Theorem \ref{t:main} relies, in fact, on modifying in a very careful way
recursive constructions of $3-(n,n,\lambda)$ directed designs, a seemingly much weaker requirement involving only $\Theta(n^3)$ demands. The recursion works by applying several operations (e.g. reversals, concatenation, composition and various label renaming) on the lines of an ordered finite affine plane, to obtain a larger
ordered affine plane. While these operations can easily keep the number of occurrences of each $3$-sequence intact, this is {\em not so} for edge pairs as these involve four vertices.
We have to apply each of our operations in the correct order
and the correct amount, as the four vertices of two disjoint edges have {\em many non-isomorphic ways} to distribute themselves among the lines of the affine plane (see Tables \ref{table:types} and \ref{table:cases}). Balancing all
distribution types in a systematic way is a delicate process of choosing the sequence and amount of operations which, fortunately, is possible (but certainly not a-priori obvious that it will work out; e.g. by the above, it is {\em impossible} to make all four sequences appear the same amount of times as subsequences).

We now turn to our second main result. One may wonder whether the ``completeness'' of $K_n$ is the reason that
its perfect separation dimension is essentially linear (and not significantly smaller). We show that this is not so. In fact, it will follow as a special case from the next theorem that almost all $n$-vertex graphs have perfect separation dimension almost linear in $n$.
Recall the Erd\H{o}s-R\'enyi random graph $G(n,p)$, the probability space of all labeled $n$-vertex graphs with
edge probability $p=p(n)$.
In what follows ``with high probability'' (w.h.p.) means with probability approaching $1$ as $n$ goes to infinity.
The next theorem shows that for a wide range of $p=p(n)$, a sampled element of $G(n,p)$ has 
perfect separation dimension close to its average degree up to a logarithmic factor. In particular, the statement
above about almost all $n$-vertex graphs follows from the case $p=\frac{1}{2}$.
\begin{theorem}\label{t:rand-lower}
	Let $p=p(n)$ satisfy $n^{-0.4}  \le p < 1$. Then w.h.p. 
	$$
	{\rm PSD}[G(n,p)] \ge \frac{np}{200 \log n}\;.
	$$
\end{theorem}
In addition to the fact that Theorem \ref{t:rand-lower} shows that perfect separation dimension of almost all graphs is essentially linear, the {\em proof} of Theorem \ref{t:rand-lower} is also of some additional value:
One of our main ingredients in the proof is the recent breakthrough of Alon et al. \cite{ABG+-2020} on the {\em minrank} of random graphs over ${\mathbb R}$ (in fact, over any field), extending an earlier result
of Golovnev, Regev, and Weinstein \cite{GRW-2018} on the minrank of random graphs over finite fields.
When combined with combinatorial, probabilistic, and linear algebraic arguments as is done in the proof of Theorem \ref{t:rand-lower}, we believe that the minrank approach will be of use for proving lower bounds of design-theoretic parameters of random graphs in other problems.

Returning to $K_n$, our final result addresses the question whether the perfection requirement is the essential cause for ${\rm PSD}(K_n)$ to be $\Omega(n)$. Perhaps if we relax the condition and just ask for a
less than perfect separating family, we can reduce the dimension considerably?
The answer too this problem turns out to be intriguingly interesting.
To formalize this question we need the following definition:
\begin{definition}[$\Delta$-balanced separation dimension]
	For a non-negative integer $\Delta$, a separating family $\F$ for a hypergraph $G$ is {\em $\Delta$-balanced} if for any two pairs of edges $\{e_1,f_1\}$, $\{e_2,f_2\}$
	where $e_1 \cap f_1 = \emptyset$ and $e_2 \cap f_2 = \emptyset$ it holds that
	$|c_{\F}(\{e_1,f_1\})-c_{\F}(\{e_2,f_2\})| \le \Delta$.
	Let $\pi_{\Delta}(G)$ be the {\em $\Delta$-balanced separation dimension}, namely the smallest possible size of a $\Delta$-balanced separating family.
\end{definition}
Observe that $\pi(G)=\pi_{\pi(G)}(G)$ and on the other extreme, $PSD(G)=\pi_{0}(G)$.
Recall again that $\pi(K_n)=\Theta(\log n)$ and in fact it is proved in \cite{BCG+-2014} that
$\pi(K_n) \le 7 \log n$.
In particular, one trivially gets that $\pi_{\Delta}(K_n) \le  7 \log n$ for $\Delta=7\log n$.
Now, suppose we decrease $\Delta$ by just a constant factor to $c \log n$ for some small absolute constant $c$.
It is reasonable to suspect that  $\pi_{\Delta}(K_n)$ should not increase by much.
Perhaps surprisingly, however, it does. Moreover, it jumps to $\Omega(n)$,
just as for $\Delta=0$ (namely, ${\rm PSD}(K_n)$).
\begin{theorem}\label{t:fair-lower}
	There exist absolute positive constants $c$ and $K$ such that for all $0 \le \Delta \le c \log n$
	$$
	\pi_{\Delta}(K_n) \ge \frac{n}{K}\;.
	$$
\end{theorem}

The rest of this paper proceeds as follows. In Section 2 we construct the $n(\log n)^{1+o(1)}$ upper bound for ${\rm PSD}(K_n)$. In Section 3 we prove the $n/2-1$ lower bound for ${\rm PSD}(K_n)$, thus completing the proof of Theorem \ref{t:main}. The perfect separation dimension of random graphs is considered in Section 4 where we prove Theorem \ref{t:rand-lower}. Finally, in Section 5 we consider $\Delta$-balanced separation dimension and prove Theorem \ref{t:fair-lower}.

\section{Constructing a perfect separating family}

In this section we prove the upper bound part of Theorem \ref{t:main}. We first set some terminology that we use throughout the rest of this section (in fact, throughout the rest of the paper). Let $S_n$ denote the symmetric group of order $n$. Hence, a separating family $\F$ of a graph with vertex set $[n]=\{1,\ldots,n\}$ is a (multi)subset of $S_n$. When explicitly specifying an element $\sigma \in S_n$ we use the notation $\sigma=(a_1,\ldots,a_n)$ to mean that
$\sigma(i)=a_i$ and $\sigma^{-1}(a_i)=i$ (this notation should not be confused with the cycle decomposition notation of a permutation).
We use $\sigma\circ\pi$ to denote the usual group element product of two elements $\sigma,\pi \in S_n$.
We also view $\sigma=(a_1,\ldots,a_n)$ as a sequence whose $i$'th element is $a_i$, hence
the notion of a subsequence of $\sigma$ is well-defined, as well as the notion of {\em reverse} where
$rev(\sigma)=(a_n,\ldots,a_1)$.

Our construction of a perfect separating family of $K_n$ requires several ingredients which we next specify. Our first ingredient is the case $k=3$ of the following definition.
\begin{definition}\label{def:sca}[perfect $k$-sequence covering array]
	A multiset $\T \subseteq S_n$ is a {\em perfect $k$-sequence covering array} if there exists a positive integer $t$ such that every sequence of $k$ distinct elements of $[n]$ is a subsequence of precisely
	$t$ elements of $\T$. We call $t$ the {\em multiplicity} of $\T$.
	The equivalent design-theoretic notion is a $k-(n,n,t)$ directed design \cite{CD-2006}.
\end{definition}
Notice that if $\T$ is a perfect $k$-sequence covering array of multiplicity $t$, then $|\T|=k!t$.
As mentioned in the introduction, a (not necessarily perfect) {\em $k$-sequence covering array} just asks for a subset of $S_n$ that has the property that each $k$-sequence of distinct elements of $[n]$ is a subsequence of {\em some} element of that subset. Fairly close lower and upper bounds are known for the minimum size of a $k$-sequence covering array of $[n]$, the state of the art for the case $k=3$ given in \cite{furedi-1996} (lower bound), \cite{tarui-2008} (upper bound)  and
for general $k$ given in \cite{radhakrishnan-2003} (lower bound), \cite{spencer-1972} (upper bound).
Fairly close upper and lower bounds are known
for the minimum size of a perfect $3$-sequence covering array of $[n]$ \cite{yuster-2020}.
Trivially, every perfect $4$-sequence covering array is
a perfect separating family. Unfortunately, as we have already mentioned in the introduction, a lower bound of $\Omega(n^2)$ is known
for the minimum size of a perfect $4$-sequence covering array \cite{yuster-2020}. Since Theorem \ref{t:main} asserts a quasi-linear upper bound for
a perfect separating family, we cannot use perfect $4$-sequence covering arrays for our construction.
Nevertheless, we will prove that a modification of the construction of perfect $3$-sequence covering arrays given in \cite{yuster-2020} (requiring a considerable modification of the proof), yields a construction of a perfect separating family.
Notice that a priori, the requirement of being a perfect separating family seems more demanding than being
a perfect $3$-sequence covering array, as the former requires satisfying $\Theta(n^4)$ demands while the latter
requires satisfying $\Theta(n^3)$ demands.

We will need the following simple lemma that merely states that a perfect separating family as well as a perfect $k$-sequence covering array is closed under relabeling.
\begin{lemma}\label{l:mult}
	Let $\pi$ be a total order of an $n$-element set $V$. Let $\F$ be a perfect separating family of $K_n$, and let $\T$ be a perfect $k$-sequence covering array.
	Then $\pi\F = \{\pi \circ \sigma\,|\, \sigma \in \F\}$ is a perfect separating family of the complete graph on vertex set $V$. Similarly,	$\pi\T = \{\pi \circ \sigma\,|\, \sigma \in \T\}$ 
	is a perfect $k$-sequence covering array of the set $V$ \footnote{Although in Definition \ref{def:sca} the ground set is $[n]$, we can equivalently consider any $n$-set $V$ as a ground set and ask for a multiset of total orders of $V$.}.
\end{lemma}
\begin{proof}
	We prove that $\pi\F$ is a perfect separating family of $K_n$ on vertex set $V$. The second part of the lemma is proved analogously.
	Let $\lambda$ be the multiplicity of $\F$. 
	Observe that $\pi \circ \sigma$ separates $(u_1,v_1)$ from $(u_2,v_2)$ if and only if
	$\sigma$ separates $(\pi^{-1}(u_1),\pi^{-1}(v_1))$ from $(\pi^{-1}(u_2),\pi^{-1}(v_2))$.
	As $\pi$ and each $\sigma \in \F$ are bijective, this implies that $\pi \F$ separates
	$(u_1,v_1)$ from $(u_2,v_2)$ precisely $\lambda$ times.
\end{proof}

A {\em finite affine plane of order $n$} is a set $\{P_1,\ldots,P_{n+1}\}$ of partitions of $[n^2]$ with the following properties: (i) Each $P_i$ contains $n$ parts of order $n$ each, denoted by
$P_{i,j}$ for $j=1,\ldots,n$. (ii) For any pair of distinct elements of $[n^2]$,
there is exactly one partition $P_i$ that contains both of them in the same part of $P_i$.
It is well-known that if $n$ is a prime power, then there is a finite affine plane of order $n$ \cite{kirkman-1850}.
In geometric terminology, the $P_{i,j}$ are called {\em lines} and the $P_i$ are called {\em parallel classes}.
In our proof we will view each $P_{i,j}$ not just as a set, but as an ordered set of $n$ elements of $[n^2]$.
In Figure \ref{figure:1} we list an affine plane of order $4$ viewed in this way. In fact, this will be the plane we use in the first step of our construction and we refer to it in various examples.
\begin{figure}[ht!]
$$
\begin{array}{c||c|c|c|c|c|}
   & P_1 & P_2 & P_3 & P_4 & P_5\\
\hline
 1 & (1,2,3,4) & (1,5,9,14) & (4,5,10,13) & (2,6,10,14) & (4,8,12,14) \\
 2 & (5,6,7,8) & (2,7,12,13) & (1,6,12,16) & (3,5,12,15) &  (1,7,10,15) \\
 3 & (9,10,11,12) & (3,8,10,16) & (2,8,9,15) & (4,7,9, 16) & (3,6,9,13) \\
 4 & (13,14,15,16) & (4,6,11,15) & (3,7, 11,14) & (1,8,11,13)& (2,5,11,16) \\
\end{array}
$$
\caption{A finite affine plane of order $4$. Row $j$ column $i$ is the line $P_{i,j}$ listed as a sequence of distinct elements.}\label{figure:1}
\end{figure}

\noindent
Our construction is established in the following lemma. However, before stating the lemma we require the definition of some integer function.
\begin{equation}
g(n) = 
\begin{cases*}
6 & if $n=3,4$ \\
2\cdot(1+r)g(r) & where $r$ is the smallest prime power satisfying $r^2 \ge n$.
\end{cases*}
\end{equation}
We note that $g(n)$ is recursively well-defined and monotone non-decreasing for every $n \ge 3$. Indeed, it is defined for $n=3,4$ and for larger $n$, we have that $r^2$ is less than $4 n$ as some integer between $n$ and $4 n - 1$ is an even power of $2$. Now, $r < 2\sqrt{n} \le n$. For example, we have $g(5)=g(9)=2(1+3)g(3)=48$,
$g(10)=g(16)=2(1+4)g(4)=60$ and $g(17)=g(25)=2(1+5)g(5)=576$. 

\begin{lemma}\label{l:main}
	For all $n \ge 3$, ${\rm PSD}(K_n) \le g(n)$. Furthermore, there is a perfect separating family of $K_n$ of size $g(n)$ that is also a perfect $3$-sequence covering array of $[n]$.
\end{lemma}
\begin{proof}
	We prove the lemma by induction on $n$. For the base cases $n=3,4$, we have $g(3)=g(4)=6$.
	Trivially, $S_3$ is a perfect $3$-sequence covering array of $\{1,2,3\}$ of size $6$. It is also vacuously a perfect separating family of $K_3$ as there are no two disjoint edges in $K_3$.
	For $n=4$, consider the subset of permutations of $S_4$ given by
	$\{1234,1432,4231,2413,3412,3214\}$. It consists of $6$ elements and it is easy to verify that it is a 
	perfect $3$-sequence covering array of $\{1,2,3,4\}$ (with multiplicity $1$) as well as a perfect separating 
	family of $K_4$ of multiplicity $2$. Notice that this set is not the smallest possible perfect separating family of $K_4$. Indeed, its subset $\{1234,1432,4213\}$ is a perfect separating family of $K_4$ of multiplicity $1$.
	
	For the general case, let us first observe that it suffices to prove the lemma for values of $n$ such that
	$n=r^2$ and $r$ is a prime power. This immediately follows from the definition $g(r)$, as if $n$ is not such,
	and $r$ is the smallest prime power with $r^2 > n$, then $g(n)=g(r^2)$ and from the fact mentioned in the introduction that if $m \ge n$, then ${\rm PSD}(K_m) \ge {\rm PSD}(K_n)$ and similarly a perfect $3$-sequence covering array of $[m]$ induces a perfect $3$-sequence covering array of $[n]$ by keeping only the integers in $[n]$ in each permutation. So, we assume until the end of the lemma that $n=r^2 \ge 9$ and $r$ is a prime power.
	
	By induction we may assume the existence of a multiset $\F$ consisting of elements of $S_r$, of size $|\F|=g(r)$ that is both a perfect separating family of $K_r$ and a perfect $3$-sequence covering array of $[r]$.
	We must show that there exists a multiset $\F^*$ of elements of $S_{r^2}$ such that
	$|\F^*|=g(r^2)=2 \cdot (1+r)g(r)$ and such that $\F^*$ is both a perfect separating family of $K_{r^2}$
	and a perfect $3$-sequence covering array of $[r^2]$.
	
	Fix some affine plane $\{P_1,\ldots,P_{r+1}\}$ of order $r$ (which exists since $r$ is a prime power) where each parallel class $P_i$ is a partition of $[r^2]$ into parts (lines) $P_{i,1}\ldots,P_{i,r}$ and recall that each $P_{i,j}$ is viewed as an ordered set. The choice of ordering is arbitrary; any total order of the elements of $P_{i,j}$ is fine.
	
	We construct $\F^*$ as follows. Each $P_i$ gives rise to a multiset $W_i$ of elements of $S_{r^2}$.
	Furthermore, $W_i$ is the union of two multisets $Y_i$ and $Z_i$, each of size $g(r)$,
	hence $|W_i|=2g(r)$. We then define $\F^*$ to be the multiset obtained by taking all the $W_i$'s.
	Hence, $\F^*$ indeed contains $g(r^2)=2(r+1)g(r)$ elements.
	
	We next show how to construct $Y_i$.
	For $\sigma \in S_r$, define $q(P_{i,j},\sigma)=P_{i,j} \circ \sigma$ to be the permutation of $P_{i,j}$ corresponding to $\sigma$ (recall, for $1 \le k \le r$,  $P_{i,j}(k)$ is the element at location $k$ of the ordered set $P_{i,j}$, so $P_{i,j} \circ \sigma$ is well-defined).
	For $\sigma \in S_r$, let $q(P_i,\sigma)$ be the concatenation of
	$q(P_{i,\sigma(1)},\sigma),\ldots,q(P_{i,\sigma(r)},\sigma)$. We call each part of this concatenation a
	{\em block}, so $q(P_i,\sigma)$ consists of $r$ blocks of size $r$ each.
	We observe that $q(P_i,\sigma) \in S_{r^2}$ and set $Y_i=\{q(P_i,\sigma)\,|\, \sigma \in \F\}$.
	Indeed, $|Y_i|=|\F|=g(r)$, as required.
	
	$Z_i$ is defined analogously to $Y_i$, but the concatenation order of the blocks is reversed.
	More formally, for $\sigma \in S_r$ let $r(P_i,\sigma)$ be the concatenation of
	$q(P_{i,\sigma(r)},\sigma),\ldots,q(P_{i,\sigma(1)},\sigma)$ (note: we reverse the concatenation order of the blocks, but the order within each block is not reversed).
	We observe that $r(P_i,\sigma) \in S_{r^2}$ and set $Z_i=\{r(P_i,\sigma)\,|\, \sigma \in \F\}$.
	Indeed, $|Z_i|=|\F|=g(r)$, as required.

	As an example, consider the case $r=4$. We use here
	$\F=\{1234,1432,4231,2413,3412,3214\}$ which we constructed initially in the case $r=4$ and use the affine plane of Figure \ref{figure:1}. Suppose, say, that $i=3$ and $\sigma=(1432) \in \F$. Then
	$P_{3,1}=(4,5,10,13)$ so $q(P_{3,1},\sigma)=(4,13,10,5)$,
	$P_{3,2}=(1,6,12,16)$ so $q(P_{3,2},\sigma)=(1,16,12,6)$,
	$P_{3,3}=(2,8,9,15)$ so $q(P_{3,3},\sigma)=(2,15,9,8)$,
	$P_{3,4}=(3,7,11,14)$ so $q(P_{3,4},\sigma)=(3,14,11,7)$.
	Concatenating them as ordered by $\sigma$, the first block is $q(P_{3,1},\sigma)$, the second block is $q(P_{3,4},\sigma)$, the third block is $q(P_{3,3},\sigma)$, and the fourth block is $q(P_{3,2},\sigma)$.
	Hence, $q(P_3,\sigma)=(4,13,10,5,3,14,11,7,2,15,9,8,1,16,12,6)$.
	The corresponding $r(P_3,\sigma)=(1,16,12,6,2,15,9,8,3,14,11,7,4,13,10,5)$ as it reverses the order of the blocks of $q(P_3,\sigma)$.
	
	It remains to show that $\F^*$ is indeed a perfect separating family of $K_{r^2}$ (Lemma \ref{l:psca} below) and is also a perfect $3$-sequence covering array of $[r^2]$ (Lemma \ref{l:separating} below). It is important to stress that, although our goal in Theorem \ref{t:main} is to construct a perfect separating family, our construction needs the fact that it is also a perfect $3$-sequence covering array in order for the inductive proof to work (just assuming inductively that
	$\F$ is a prefect separating family of $K_r$ will not suffice to prove that $\F^*$ is such for $K_{r^2}$). Thus, given Lemmas \ref{l:psca} and \ref{l:separating}, the proof of Lemma \ref{l:main} is completed.
	\end{proof}
	
	\begin{lemma}\label{l:psca}
		$\F^*$ is a perfect $3$-sequence covering array of $[r^2]$.
	\end{lemma}
	\begin{proof}
	The lemma's proof is very similar to the proof of the construction given in \cite{yuster-2020}, but we repeat it for completeness (note, however, that the proof that $\F^*$ is a perfect separating family given later is more involved and is not needed in the proof given in \cite{yuster-2020}). So, let $a,b,c$ be three distinct integers in $[r^2]$.
	We must prove that $a,b,c$ is a subsequence of precisely $|\F^*|/6=g(r^2)/6$ elements of $\F^*$.
	There are two cases to consider: (i) $\{a,b,c\}$ is contained in some $P_{i,j}$.
	(ii) $\{a,b,c\}$ is not a subset of any $P_{i,j}$.
	
	Case (i): Suppose $\{a,b,c\}$ appears in some $P_{i,j}$.  Recall that $\F$ is a perfect $3$-sequence covering array of multiplicity $|\F|/6=g(r)/6$. In each $q(P_i,\sigma) \in Y_i$, all $\{a,b,c\}$ appear in the same block, namely block number $\sigma^{-1}(j)$. Inside the block, the internal order of $\{a,b,c\}$ is
	their order in $q(P_{i,j},\sigma)$. But by Lemma \ref{l:mult}, $P_{i,j}\F=\{q(P_{i,j},\sigma)\,|\,\sigma \in \F\}$
	is a perfect $3$-sequence covering array of $P_{i,j}$. Hence, $a,b,c$ is a subsequence of precisely
	$g(r)/6$ elements of $Y_i$. Similarly, $a,b,c$ is a subsequence of precisely $g(r)/6$ elements of $Z_i$
	(recall that the internal block order of each block of $r(P_i,\sigma)$ is the same as the internal block order of each block of $q(P_i,\sigma)$).
	Altogether $a,b,c$ is a subsequence of $g(r)/3$ elements of $W_i$.
	Now, if $i' \neq i$, then each of $a,b,c$ appears in a distinct part of $P_{i'}$ (recall, in an affine plane no pair appears twice in any of the parts). But in $q(P_{i'},\sigma)$ the order of these parts is determined by $\sigma$. Since $\F$ is a perfect $3$-sequence covering array, precisely in $g(r)/6$ of the elements of $Y_{i'}$, the part containing $a$ appears before the part containing $b$ and the part containing $b$ appears before the part containing $c$. Hence $a,b,c$ is a subsequence of precisely
	$g(r)/6$ elements of $Y_{i'}$. The same argument holds for $Z_{i'}$ as the parts just reverse their order (so $c$ appears before $b$ and $b$ appears before $a$ in $r(P_{i'},\sigma)$ if and only if $a$ is before $b$ and $b$ is before $c$ in $q(P_{i'},\sigma)$. Altogether we have that $a,b,c$ is a subsequence of precisely
	$(r+1)g(r)/3=g(r^2)/6$ elements of $\F^*$.
	
	Case(ii): Suppose $\{a,b,c\}$ is not a subset of any $P_{i,j}$. Let $\gamma$ be the unique index such that $\{a,b\}$ is a subset of some part of $P_{\gamma}$, let $\beta$ be the unique index such that $\{a,c\}$ is a subset of some part of $P_{\beta}$ and let $\alpha$ be the unique index such that $\{b,c\}$ is a subset of some part of $P_{\alpha}$. Note that $\alpha,\beta,\gamma$ are indeed unique and distinct as follows from the properties of an affine plane. As in the previous case, we have that if $i \notin \{\alpha,\beta,\gamma\}$ then $a,b,c$ is a subsequence of $g(r)/3$ elements of $W_i$.
	How many times is $a,b,c$ a subsequence in $W_\beta$? The answer is $0$, since in each element of $W_\beta$, $a$ and $c$ appear in the same block while $b$ appears in another block.
	How many times is $a,b,c$ a subsequence of $W_\alpha$? Since $b,c$ are in the same block of each element of $W_\alpha$ and since in precisely half of the elements of each of $Y_\alpha$ and  $Z_\alpha$, $b$ appears before $c$ (we use here the fact that a perfect $3$-sequence covering array is trivially also a perfect $2$-sequence covering array) we have that if $\sigma$ is such that in $q(P_\alpha,\sigma)$, $b$ appears before $c$, then in precisely one of $q(P_\alpha,\sigma)$ or $r(P_\alpha,\sigma)$, $a,b,c$ is a subsequence.
	Hence, $a,b,c$ is a subsequence of $g(r)/2$ elements of $W_\alpha$. The same argument holds for $W_\gamma$.
	Altogether we have that $a,b,c$ is a subsequence of precisely $(r-2)g(r)/3+0+2g(r)/2=(r+1)g(r)/3=g(r^2)/6$ elements of $\F^*$.
	\end{proof}

\begin{lemma}\label{l:separating}
	$\F^*$ is a perfect separating family of $K_{r^2}$.
\end{lemma}
\begin{proof}
	Consider two disjoint edges
	$ab$ and $cd$ of $K_n$. We must show that precisely $|\F^*|/3=g(r^2)/3$ elements of $\F^*$ separate them.
	Here there are more cases to consider, according to the distribution of the $6$ pairs induced by $\{a,b,c,d\}$ among
	the parts of the affine plane. To effectively examine these cases, it is beneficial to classify how a parallel class $P_i$ of the affine plane distributes $a,b,c,d$ among its lines. This classification is made according to {\em types}, as given in Table \ref{table:types}. As can be seen from that table, there are seven possible types, up to isomorphism. A parallel class $P_i$ is of type $1$ if some line of it contains all of $a,b,c,d$.
	It is of type $2$ if some line contains three of $\{a,b,c,d\}$ and some other line contains the fourth
	vertex (in Table \ref{table:types} we write for type $2$ that $\{a,b,c\}$ is the triple contained in one line, but this is isomorphic to any other triple). It is of type $3$ if the edge $ab$ is in one line and the edge $cd$ is in another line. It is of type $4$ if precisely one of the edges $ab$ or $cd$ is in some line, and the other two vertices are each in a separate line. It is of type $5$ if two pairs are in two lines and neither of these pairs are the edges $ab$ or $cd$. It is of type $6$ if one pair is in some line and this pair is not one of the edges $ab$ or $cd$, and the other two vertices are each in a separate line. Finally, it is of type $7$ if each vertex is in a separate line.
	
	\begin{table}[ht!]
	\begin{center}
		\begin{tabular}{c||c|c} 
			type & distribution of $\{a,b,c,d\}$ among parts & \# of times $ab,cd$ separated in $W_i$\\
			\hline
			$1$ & $\{a,b,c,d\}$ & $2g(r)/3$\\
			$2$ & $\{a,b,c\}$ $\{d\}$ & $2g(r)/3$\\
			$3$ & $\{a,b\}$ $\{c,d\}$ & $2g(r)$ \\
			$4$ & $\{a,b\}$ $\{c\}$ $\{d\}$ & $4g(r)/3$\\
			$5$ & $\{a,c\}$ $\{b,d\}$ & $0$ \\
			$6$ & $\{a,c\}$ $\{b\}$ $\{d\}$ & $g(r)/3$ \\
			$7$ & $\{a\}$ $\{b\}$ $\{c\}$ $\{d\}$ & $2g(r)/3$\\
		\end{tabular}
	\end{center}
	\caption{The possible types (up to isomorphism) of a parallel class of an affine plane with respect to the two disjoint edges $ab,cd$. For each type, the number of times $ab,cd$ are separated in $W_i$ is given.}\label{table:types}
\end{table}
	
	To illustrate these types, consider the affine plane given in Figure  \ref{figure:1} and the parallel class $P_1$.
	Then $P_1$ is of type $1$ with respect, say, to the disjoint edges $12,34$. It is of type $2$ w.r.t. $12,35$. It is of type $3$ w.r.t. $12,56$. It is of type $4$ w.r.t. $12,59$. It is of type $5$ w.r.t. $15,26$. It is of type $6$ w.r.t. $15,29$. It is of type $7$ w.r.t. the disjoint edges whose union is, say, $\{1,5,9,13\}$.

	Having described the types, we can now consider all possible cases for a pair of disjoint edges $ab,cd$.
	For each of the $r+1$ parallel classes of the affine plane, we consider its type with respect to $ab,cd$,
	and count how many parallel classes are from each type. This gives a $7$-tuple of non-negative integers
	summing up to $r+1$, where the $j$'th tuple coordinate is the number of parallel classes of type $j$. Now,
	each possible $7$-tuple corresponds to a distinct {\em case}. The possible cases, together with their
	corresponding $7$-tuple, are given in Table \ref{table:cases}. As can be seen, there are eight
	possible cases (not all cases are necessarily realizable in every affine plane). For example, a pair of disjoint edges $ab,cd$ belongs to Case 1, if some parallel class has
	type 1 with respect to it (namely, all of $a,b,c,d$ are in the same line), and all other $r$ parallel classes are necessarily of type $7$, namely each of $a,b,c,d$ is in a distinct line in each of these classes.
	
	To illustrate some cases, consider again the affine plane given in Figure  \ref{figure:1}.
	The pair of edges $12,34$ belongs to Case 1, as $P_1$ is of type $1$ with respect to it and $P_2,P_3,P_4,P_5$
	are each of type $7$ with respect to it. The pair $12,35$ belongs to Case $2$, as $P_1$ is of type $2$ with respect to it, $P_4$ is of type $4$, $P_2,P_5$ are of type $6$, and $P_3$ is of type $7$.
	Similarly, one can check that the pair of edges $12,57$ belongs to Case $3$, the pair of edges $12,56$ belongs to Case $5$ and the pair $13,89$ to Case $7$.

\begin{table}[ht!]
	\begin{center}
		\begin{tabular}{c||c|c|c|c|c|c|c|} 
			case & type $1$ & type $2$ & type $3$ & type $4$ & type $5$ & type $6$ & type $7$ \\
			\hline
			$1$ & $1$ &  &  &  &  &  & $r$ \\
			$2$ &  & $1$  &  & $1$ &  & $2$ & $r-3$ \\
			$3$ &  &  & $1$  &  & $2$ &  & $r-2$ \\
			$4$ &  &  & $1$ &  & $1$ & $2$  & $r-3$ \\
			$5$ &  &  & $1$ &  &  & $4$ & $r-4$ \\
			$6$ &  &  &  & $2$ & $2$ &  & $r-3$ \\
			$7$ &  &  &  & $2$ & $1$ & $2$ & $r-4$ \\
			$8$ &  &  &  & $2$ &  & $4$ & $r-5$ \\
		\end{tabular}
	\end{center}
	\caption{The possible cases of how the pairs induced by $\{a,b,c,d\}$ given two disjoint edges $ab,cd$ are distributed among the parallel classes of an affine plane. For each type, the number of parallel classes of this type is listed, (empty cells correspond to zero).}\label{table:cases}
\end{table}

	Suppose now that $ab,cd$ is a pair of disjoint edges and $P_i$ is some parallel class. We compute
	the number of elements of $W_i$ that separate $ab$ and $cd$. We do this according to $P_i$'s type w.r.t.
	$ab,cd$.
	
	{\em $P_i$ is of type $1$}.
	In this case, $\{a,b,c,d\}$ appears in some $P_{i,j}$.  Recall that $\F$ is a perfect separating family of multiplicity $|\F|/3=g(r)/3$. In each $q(P_i,\sigma) \in Y_i$, all $\{a,b,c,d\}$ appear in the same block, namely block number $\sigma^{-1}(j)$. Inside the block, the internal order of $\{a,b,c,d\}$ is
	their order in $q(P_{i,j},\sigma)$. But by Lemma \ref{l:mult}, $P_{i,j}\F=\{q(P_{i,j},\sigma)\,|\,\sigma \in \F\}$
	is a perfect separating family of $P_{i,j}$. Hence, $ab$ and $cd$ are separated by precisely
	$g(r)/3$ elements of $Y_i$. Similarly, $a,b,c,d$ are separated by precisely $g(r)/3$ elements of $Z_i$
	(recall that the internal block order of each block of $r(P_i,\sigma)$ is the same as the internal block order of each block of $q(P_i,\sigma)$).
	Altogether $ab$ and $cd$ are separated by $2g(r)/3$ elements of $W_i$.
	
	{\em $P_i$ is of type $2$}. Without loss of generality, $\{a,b,c\}$ appears in some $P_{i,j}$ and $d$ appears in
	some $P_{i,k}$ with $k \neq j$. Since $\F$ is a perfect $3$-sequence covering array of multiplicity $g(r)/6$, in precisely $4g(r)/6$ of the elements of $q(P_i,\sigma) \in Y_i$, the block containing $a,b,c$ does
	not have $c$ between $a$ and $b$ (so the subsequence on $\{a,b,c\}$ is one of $(a,b,c),(b,a,c),(c,a,b),(c,b,a)$).
	For each such element, precisely one of $q(P_i,\sigma)$ or $r(P_i,\sigma)$, separates $ab$ from $cd$
	(for example, if the subsequence on $\{a,b,c\}$ is $(b,a,c)$, we would like the block containing $d$ to appear after the block containing $a,b,c$, and this happens in precisely one of $q(P_i,\sigma)$ or $r(P_i,\sigma)$).
	Altogether $ab$ and $cd$ are separated by $4g(r)/6=2g(r)/3$ elements of $W_i$.
	
	{\em $P_i$ is of type $3$}. Here $\{a,b\}$ is in some $P_{i,j}$ and $\{c,d\}$ are in some $P_{i,k}$ with $k \neq j$. So each element of $W_i$ separates $a,b$ from $c,d$. Altogether $ab$ and $cd$ are separated by $|W_i|=2g(r)$ elements of $W_i$.
	
	{\em $P_i$ is of type $4$}. Here $\{a,b\}$ is in some $P_{i,j}$, $c \in P_{i,k}$ and $d \in P_{i,\ell}$ where $j,k,\ell$ are distinct. In order for $a,b$ to be separated from $c,d$ in $q(P_i,\sigma) \in Y_i$, we would like the block containing $a,b$ to not be in between the block containing $c$ and the block containing $d$.
	But since $\F$ is a $3$-sequence covering array, this occurs in precisely $2g(r)/3$ elements of $Y_i$,
	and the same holds for $Z_i$. Altogether $ab$ and $cd$ are separated by $4g(r)/3$ elements of $W_i$.
	
	{\em $P_i$ is of type $5$}. Here $\{a,c\}$ is in some $P_{i,j}$ and $\{b,d\}$ is in some $P_{i,k}$ with $k \neq j$. In this case, no element of $W_i$ separates $ab$ from $cd$ as $a,b$ are always in the same block and $c,d$ are always in the same block of each element of $W_i$.
	
	{\em $P_i$ is of type $6$}. Here $\{a,c\}$ is in some $P_{i,j}$, $b \in P_{i,k}$ and $d \in P_{i,\ell}$ where $j,k,\ell$ are distinct. The only way that $ab$ and $cd$ could be separated by an element of $W_i$ is if the block containing $a,c$ is in between the two blocks containing $b$ and $d$. As $\F$ is a perfect $3$-sequence covering array, this occurs in precisely $g(r)/3$ of the elements of $Y_i$. For each such element
	$q(P_i,\sigma)$, if $a$ is before $c$, we would like the block containing $b$ to be before the block containing $d$ (recall that the block containing $a,c$ is in between them). If this does not occur in $q(P_i,\sigma)$
	then it necessarily occurs in $r(P_i,\sigma)$ and vice versa. Similarly, if $c$ is before $a$
	we would like the block containing $d$ to be before the block containing $b$. If this does not occur in $q(P_i,\sigma)$ then it necessarily occurs in $r(P_i,\sigma)$ and vice versa.
	Altogether $ab$ and $cd$ are separated by $g(r)/3$ elements of $W_i$.
	
	{\em $P_i$ is of type $7$}. Here each of $a,b,c,d$ is in a distinct block. Since $\F$ is a perfect separating set, we have that in $Y_i$, the two blocks containing $a$ and $b$ are separated from the blocks
	containing $c$ and $d$ precisely $g(r)/3$ times. The same holds for $Z_i$. 
	Altogether $ab$ and $cd$ are separated by $2g(r)/3$ elements of $W_i$.
	
	We have summarized the analysis above in the third column of Table \ref{table:types}.
	It is now immediate to verify that every pair $ab,cd$ of disjoint edges is separated precisely
	$|F^*|/3=g(r^2)/3$ times in $\F^*$. Indeed, for each possible case from Table  \ref{table:cases}, if we take the scalar product of its corresponding $7$-tuple given in Table \ref{table:cases} and the third column of Table \ref{table:types}, we obtain $2(r+1)g(r)/3=g(r^2)/3$. For instance, in Case $7$ we obtain
	$g(r)/3$ times the scalar product $(0,0,0,2,1,2,r-4) \cdot (2,2,6,4,0,1,2)$ giving $(2 r+2)g(r)/3$.
	The other cases are similarly verified.
\end{proof}

The next lemma is a consequence of Lemma \ref{l:main} and is the upper bound part of Theorem \ref{t:main}.
\begin{lemma}
	Let $\epsilon > 0$. There is a constant $C=C(\epsilon)$ such that for all $n \ge 3$,
	${\rm PSD}(K_n) \le Cn(\log n)^{1+\epsilon}$.
\end{lemma}
\begin{proof}
	We require the following result of Baker et al. \cite{BHP-2001} which 
	states that there is always a prime strictly between $x$ and $x+O(x^{21/40})$.
	In particular, for all sufficiently large $n$, there is a prime between $\sqrt{n}$ and $\sqrt{n}+n^{1/3}$.
	
	Given $\epsilon > 0$, let $\delta = 2^{\epsilon/4}-1$.
	Let $N$ be the least integer for which for all $n \ge N$, the following hold.
	\begin{enumerate}
	\item[(i)]
	$\sqrt{n}+n^{1/3}+1 \le (1+\delta)\sqrt{n}$.
	\item[(ii)]
	If $r$ is the smallest prime power satisfying $r^2 \ge n$ then $\sqrt{n} \le r \le \sqrt{n}+n^{1/3}$.
	\item[(iii)]
	$(\log(\sqrt{n}+n^{1/3}))^{1+\epsilon} \le (\log n)^{1+\epsilon}/2^{1+\epsilon/2}$.
	\end{enumerate}

	Let $C=g(N)$ and observe that indeed $C=C(\epsilon)$.
	By Lemma \ref{l:main} it suffices to prove that $g(n) \le Cn(\log n)^{1+\epsilon}$.
	Indeed, this clearly holds for all $3 \le  n \le N$ since $g(n) \le g(N)=C$ in this case.
	Assume now that $n > N$ and assume inductively that the claim holds for values smaller than $n$.

	Let $r$ be the smallest prime power satisfying $r^2 \ge n$.
	So, we have that $\sqrt{n} \le r \le \sqrt{n}+n^{1/3}$.
	Now, by Lemma \ref{l:main} we have that
	$g(n) = g(r^2) =2(r+1)g(r)$.
	By the induction hypothesis,
	\begin{align*}
	g(n) & \le  2(r+1)Cr(\log r)^{1+\epsilon}\\
	& \le 2(\sqrt{n}+n^{1/3}+1)C(\sqrt{n}+n^{1/3})(\log (\sqrt{n}+n^{1/3}))^{1+\epsilon}\\
	& \le 2(1+\delta)^2 n C (\log n)^{1+\epsilon}/2^{1+\epsilon/2}\\
	& = (1+\delta)^2 n C (\log n)^{1+\epsilon}/2^{\epsilon/2}\\
	& = Cn(\log n)^{1+\epsilon}\;.
	\end{align*}
\end{proof}

\section{A lower bound for a perfect separating family of $K_n$}

The following lemma establishes the lower bound in Theorem \ref{t:main}.
\begin{lemma}\label{l:lower}
	Let $n \ge 4$ and let $\F$ be a perfect separating family of $K_n$. Then $|\F| \ge n/2-1$.
\end{lemma}
\begin{proof}
	As the lemma holds trivially for $n=4$, we assume that $n \ge 5$.
	For a permutation $\sigma=(a_1,\ldots,a_n)$ of vertices of $K_n$, we say that the pairs $\{a_1,a_2\}$ and $\{a_{n-1},a_n\}$ are extremal in $\sigma$. Let $\lambda=|\F|/3$ be the multiplicity of $\F$.
	We say that a pair $\{u,v\}$ of vertices is {\em extremal} in $\F$ if $\{u,v\}$ is extremal in $\lambda$ elements of $\F$. As each $\sigma \in \F$ has two pairs that are extremal in $\sigma$, the overall count of extremal pairs is $2|\F|=6\lambda$, so there can be at most $6$ pairs that are extremal in $\F$.
	But since $n \ge 5$, there are more than $6$ pairs of vertices so let $\{u,v\}$ be a pair of vertices that is {\em not} extremal in $\F$. We may assume by relabeling that $\{u,v\}=\{n-1,n\}$.

We next define several matrices over ${\mathbb R}$.
Let $A$ be the matrix whose rows are indexed by $[n-2]$ (i.e. all vertices except $u,v$) and whose columns are
indexed by $\F$. Set $A[w,\sigma]=1$ if  $\sigma^{-1}(w) < \min \{\sigma^{-1}(u),\sigma^{-1}(v)\}$.
Otherwise, set $A[w,\sigma]=0$. Let $B$ be the matrix whose rows are indexed by $[n-2]$ and whose columns are
indexed by $\F$. Set $B[w,\sigma]=1$ if  $\sigma^{-1}(w) > \max \{\sigma^{-1}(u),\sigma^{-1}(v)\}$.
Otherwise, set $B[w,\sigma]=0$. Clearly,
$$
rank(AA^T) \le rank(A) \le |\F| ~,~~~ rank(BB^T) \le rank(B) \le |\F|\;. 
$$
Let $C=AA^T+BB^T$. So, $C$ is a square matrix of order $n-2$. We prove that $C$ is non-singular.
Notice that this suffices since if so, we have that
\begin{equation}\label{e:triv}
rank(C) \le rank(AA^T)+rank(BB^T) \le 2|\F|
\end{equation}
implying that $|\F| \ge n/2-1$.

So it remains to prove that $C$ is non-singular.
Consider first some off-diagonal entry of $AA^T$, say $(AA^T)[x,y]$. Then this entry counts the number of elements of $\sigma \in \F$ in which both vertices $x$ and $y$ appear in $\sigma$ before $u,v$.
Similarly, $(BB^T)[x,y]$ counts the number of elements of $\sigma \in \F$ in which both vertices $x$ and $y$ appear in $\sigma$ after $u,v$. Therefore, since $\F$ is a perfect separating family, we have that
$(AA^T)[x,y]+(BB^T)[x,y]=\lambda$. Hence, all the off-diagonal entries of $C$ equal $\lambda$.

A diagonal entry of $C$, say $C[w,w]$ equals the number of elements $\sigma \in \F$ in which $w$ is either before both $u,v$ or after both $u,v$. But notice that if $w' \in [n]-\{u,v,w\}$ then $w'w$ is separated from $uv$ precisely $\lambda$ times, and we therefore have, in particular, that $C[w,w] \ge \lambda$.
We claim, however, that $C[w,w] > \lambda$.
Assume otherwise, that $C[w,w]=\lambda$. So, there is a subset $\F_w \subset \F$ of precisely $\lambda$ elements of $\F$ in which $w$ is either before both $u$ and $v$ or after both $u$ and $v$.
Consider again some $w' \in [n]-\{u,v,w\}$.
For any $\sigma \in F \setminus \F_w$, we have that $w$ is in between $u$ and $v$ and therefore $w'w$ is not separated from $uv$ by $\sigma$.
As $w'w$ and $uv$ need to be separated $\lambda$ times, this implies that precisely all elements of $\F_w$ separate
$w'w$ and $uv$. Hence, $w'$ and $w$ must both either be before both $u$ and $v$ or else both be after both $u$ and $v$. As this holds for all $w' \in [n]-\{u,v,w\}$, this means that $\{u,v\}$ is extremal in each $\sigma \in F_w$.
Hence $\{u,v\}$ is extremal in $\F$ precisely $\lambda$ times, contradicting the choice of $\{u,v\}$.
We have proved that $C[w,w] > \lambda$. for all $w \in [n-2]$.

To summarize, $C$ is a matrix whose off-diagonal entries are all equal to $\lambda$ and whose diagonal entries are all greater than $\lambda$. It is an easy exercise that such matrices are non-singular.

\end{proof}

\section{Perfect separation dimension of random graphs}

The following definition plays a central role in the proof of Theorem \ref{t:rand-lower}.
\begin{definition}[minrank]\label{def:minrank}
The {\em minrank} of a graph $G$ on vertex set $[n]$ over a field ${\mathbb F}$, denoted
${\rm minrank}_{\mathbb F}(G)$,
is the smallest possible rank of a matrix $M \in {\mathbb F}^{n \times n}$ with nonzero diagonal entries
such that $M[u,v] = 0$ for any pair $u,v$ of distinct nonadjacent vertices of $G$.
\end{definition}
The notion of minrank (over ${\mathbb R}$) was first studied by Lov\'asz \cite{lovasz-1979} and has found several applications, see \cite{ABG+-2020,GRW-2018}. We need the following result of Alon et al. \cite{ABG+-2020} regarding the minrank of a random graph over ${\mathbb R}$ (in fact, over every field), extending an earlier similar result over finite fields of Golovnev, Regev, and Weinstein \cite{GRW-2018}.
\begin{lemma}[\cite{ABG+-2020}]\label{l:minrank}
	Assume that $q = q(n)$ satisfies $1/n \le q \le 1$. Then w.h.p.
	$$
	{\rm minrank}_{\mathbb R}(G(n,q)) \ge \frac{n \log(1/q)}{80\log n}\;.
	$$
\end{lemma}

Before proving Theorem \ref{t:rand-lower}, we need to establish some properties of $G(n,p)$ that hold w.h.p.
\begin{lemma}\label{l:prop1}
	Let $t=\lceil 3\log n/p \rceil$. It holds w.h.p. in $G(n,p)$ that for any two disjoint vertex sets $A,B$ of order $t$ each, there is an edge with one endpoint in $A$ and one endpoint in $B$.
\end{lemma}
\begin{proof}
	Consider two disjoint vertex sets $A,B$ of order $t$ each. The probability that no edge connects a vertex of $A$ with a vertex of $B$ is $(1-p)^{t^2}$. As there are less than $n^{2 t}$ choices for $A$ and $B$, the probability that there is some pair violating the stated property is at most
	$$
	n^{2 t} (1-p)^{t^2} \le n^{2 t}\left(1-\frac{3\log n}{t}\right)^{t^2} \le n^{-t} = o(1)\;.
	$$
\end{proof}
\begin{lemma}\label{l:prop2}
	W.h.p. it holds that any set $X$ of at least $9\log n/p$ vertices of $G(n,p)$ induces a connected subgraph of order at least $|X|/3$.
\end{lemma}
\begin{proof}
	Let $X$ be a set of vertices with $|X| =x \ge 9\log n/p$. Assume that $X$ does not induce a connected subgraph of order at least $x/3$. Then we can partition $X$ into two parts $Y$ and $X \setminus Y$ with
	$2x/3 \ge |Y|\ge |X-Y| \ge x/3$ and there is no edge between $X \setminus Y$ and $Y$. But this implies, in particular, that we have $A \subset Y$ and $B \subseteq X \setminus Y$ with $|A|=|B|=t=\lceil 3\log n/p \rceil$ and no edge connecting them,
	and we have just proved in Lemma \ref{l:prop1} that w.h.p. this does not occur in $G(n,p)$.
\end{proof}

For a permutation $\sigma$ of vertices of an $n$-vertex graph $G$, we say that a pair of vertices $\{u,v\}$
is {\em $k$-extremal in $\sigma$} if $\max \{\sigma^{-1}(u)\,,\sigma^{-1}(v)\} \le k$ or
$\min \{\sigma^{-1}(u)\,,\sigma^{-1}(v)\} \ge n-k+1$.
Let $\F$ be a perfect separating family of $G$, with multiplicity $\lambda$.
We say that an edge $uv$ of $G$ is {\em $k$-extremal in $\F$} if $\{u,v\}$ is $k$-extremal in at least $\lambda$ elements of $\F$.
\begin{lemma}\label{l:prop3}
Assume that $p = p(n)$ satisfies $n^{-0.4} \le p < 1$. Then w.h.p.
the following hold for every perfect separating family $\F$ of $G \sim G(n,p)$:
Either $|\F| \ge n p/4$ or else there is an edge $uv$ of $G$ such that $uv$ is not $\lfloor \sqrt{n} \rfloor$-extremal in $\F$.
\end{lemma}
\begin{proof}
	Assume that $|\F| < n p/4$.
	For each $\sigma \in \F$, there are at most $2\binom{\lfloor \sqrt{n} \rfloor}{2} < n$ pairs
	that are $\lfloor \sqrt{n} \rfloor$-extremal in $\sigma$.
	The overall count of $\lfloor \sqrt{n} \rfloor$-extremal pairs is therefore less than $n|\F| < n^2 p/4$.
	But the expected number of edges in $G$ is $p\binom{n}{2} > p n^2/3$ and distributed binomially ${\mathcal B}(\binom{n}{2},p)$, so w.h.p. $G$ has more than $n^2 p/4$ edges. Hence, w.h.p. there is an edge $uv$
	of $G$ such that $uv$ is not $\lfloor \sqrt{n} \rfloor$-extremal in $\F$.
\end{proof}
Notice that the last lemma holds also for values of $p$ significantly smaller than $n^{-0.4}$, but as later considerations require this assumption, we preferred to state the lemma in this way.

\vspace*{4mm}
\noindent {\em Proof of Theorem \ref{t:rand-lower}.}
According to the theorem's statement, we assume that  $p=p(n)$ satisfies $n^{-0.4} \le p < 1$.
We may also assume that $p \le 1-1/n$ as otherwise, each vertex of $G(n,p)$ has probability
at least $1/e$ of having degree $n-1$, so (e.g. using Chebyshev's inequality) w.h.p. $G(n,p)$
has a clique of order at least $n/4$ and the result follows from Theorem \ref{t:main}.
Hereafter, $n^{-0.4} \le p < 1-1/n$ so $1/n \le q=1-p \le 1-n^{0.4}$.

Let $G \sim G(n,p)$ and let $\F$ be a perfect separating family of $G$ with multiplicity $\lambda$.
By Lemmas \ref{l:prop1}, \ref{l:prop2}, \ref{l:prop3} we may assume that $G$ satisfies the statements of each of these lemmas. Furthermore, $\overline{G}$ (the complement of $G$) is chosen uniformly from $G(n,q)$ so
we may also assume by Lemma \ref{l:minrank} that
${\rm minrank}_{\mathbb R}(\overline{G}) \ge \frac{n \log(1/q)}{80\log n}$.
Finally, if $|\F| \ge n p/4$ then the statement of the theorem holds, so we may assume that $|\F| < n p/4$
and hence by Lemma \ref{l:prop3}, there is an edge $uv$ of $G$ such that $uv$ is not
$\lfloor \sqrt{n} \rfloor$-extremal in $\F$. We may assume by relabeling that $\{u,v\}=\{n-1,n\}$.
We prove that ${\mathrm{PSD}}(G) \ge \frac{np}{200 \log n}$.

As in the proof of Lemma \ref{l:lower}, we consider the matrices $A,B,C=AA^T+BB^T$ defined there.
We will prove that $rank(C) \ge \frac{np}{100\log n}$ and obtain using (\ref{e:triv}) that $|\F| \ge \frac{np}{200\log n}$ as required.
First observe that for each edge $xy$ of $G$ it holds that $C[x,y]=\lambda$. However, unlike the proof of Lemma \ref{l:lower},
we have no information regarding the off-diagonal entries of $C$ that correspond to non-edges of $G$ (equivalently,
edges of ${\overline G}$). We next consider diagonal entries of $C$. We first claim that $C[w,w] \ge \lambda$
for each $w \in [n-2]$. To see this, observe that the expected minimum degree of $G$ is $\Theta(np)$ and hence
w.h.p. the minimum degree of $G$ is at least $3$, so we may assume this is the case. But this implies that $w$ has some neighbor $w' \notin \{u,v\}$. As the edge $w'w$ is separated from $uv$ precisely $\lambda$ times, and as each $\sigma \in \F$ which separates them contributes to $C[w,w]$ (as $w$ must either be before both $u,v$ in $\sigma$ or after both $u,v$ in
$\sigma$), we have that $C[w,w] \ge \lambda$. However, it would also be necessary to prove that many diagonal entries are, in fact, larger than $\lambda$ which is what we show next.

Let $W \subseteq [n-2]$ be the set of vertices with $w \in W$ if $C[w,w]=\lambda$. We show that
$|W| \le 9\log n/p$. Assume otherwise, then by Lemma \ref{l:prop2}, there is a subset $W' \subset W$
with $|W'| = t=\lceil 3\log n/p \rceil$ such that $W'$ induces a connected subgraph in $G$.
Consider some $w \in W'$ and let $\F_w \subseteq \F$ consist of all $\sigma \in \F$ such that
$w$ is either before both $u,v$ in $\sigma$ or after both $u,v$ in $\sigma$. Then, by the definition of
$W$, we have that $|\F_w| = \lambda$. Now, let $w' \in W'$ be any neighbor of $w$.
Then $w'w$ is separated from $uv$ precisely by the elements of $\F_w$, and since $w' \in W$ this implies that
$\F_w=\F_{w'}$. But since $W'$ induces a connected subgraph, we have that $\F_w=F_{w'}$ for any pair of
distinct elements of $W'$ so we can actually denote this common subset by $F_{W'}$ as it is independent of the choice of $w \in W'$.
Now let $W^* \subseteq [n-2]$ be the set of vertices having a neighbor in $W'$
(in particular $W' \subseteq W^*$). It follows that for each $\sigma \in \F_{W'}$, all vertices in $W^*$ are either all
before $u,v$ in $\sigma$ or all after both $u,v$ in $\sigma$. It now follows from Lemma \ref{l:prop1} that
$[n-2] \setminus W^*$ contains less than $t$ elements as otherwise $A=W'$ and $B \subset [n-2] \setminus W^*$
with $|B|=t$ violate the statement of the lemma. But this, in turn, implies that $\{u,v\}$ is $(t+1)$-extremal
in $\sigma$. Indeed, if all elements of $W^*$ are after $u,v$ in $\sigma$ then
we have $\max \{\sigma^{-1}(u)\,,\sigma^{-1}(v)\} \le t+1$ and if all elements of $W^*$ are before $u,v$ in $\sigma$ then we have $\min \{\sigma^{-1}(u)\,,\sigma^{-1}(v)\} \ge n-t$. Now, since
$t+1=\lceil 3\log n/p \rceil+1 \le \lfloor \sqrt{n} \rfloor$ (here we use that $p \ge n^{-0.4}$) we have
that $\{u,v\}$ is $\lfloor \sqrt{n} \rfloor$-extremal in $\sigma$. As this holds for each $\sigma \in \F_{W'}$,
and since $|F_{W'}|=\lambda$, we have that $\{u,v\}$ is $\lfloor \sqrt{n} \rfloor$-extremal in $\F$,
contradicting our choice of $u,v$. We have proved that $|W| \le 9\log n/p$.

Consider now the matrix $D=C-\lambda J$ where $J$ is the all-one matrix.
Then we have that for each edge $xy$ of $G$, equivalently, each non-edge $xy$ of $\overline{G}$, it holds that
$D[x,y]=0$. Furthermore, at most  $|W| \le 9\log n/p$ diagonal entries of $D$ are zero.
Arbitrarily set the zero diagonal entries of $D$ to nonzero values, obtaining a matrix $D^*$.
As $D$ and $D^*$ differ in at most $9\log n/p$ entries, we have that $rank(D) \ge rank(D^*)-9\log n/p$.
Furthermore, by Lemma \ref{l:minrank} applied to $\overline{G}$ we have that $rank(D^*) \ge \frac{n \log(1/q)}{80\log n}$. We therefore have that
\begin{align*}
rank(C) & \ge rank(D)-1\\
& \ge rank(D^*)-\frac{9\log n}{p}-1\\
& \ge \frac{n \log(1/q)}{80\log n} - \frac{9\log n}{p}-1\\
& \ge \frac{n \log(1+p)}{80\log n} - \frac{9\log n}{p}-1\\
& \ge \frac{np}{80\log n} - \frac{9\log n}{p}-1\\
& \ge \frac{np}{100\log n}
\end{align*}
where in the last inequality we have used that $p \ge n^{-0.4}$.
\qed

\section{$\Delta$-balanced separation dimension}

Here we prove Theorem \ref{t:fair-lower}. An important ingredient in our proof is the following result of Alon \cite{alon-2009}.
\begin{lemma}[\cite{alon-2009}]\label{l:perturbed}
	There exists an absolute positive constant $\mu$ so that the following holds. Let $D$ be an
	$n$ by $n$ real matrix with $| D[i,i]| \ge \frac{1}{2}$ for all $i$ and
	$|D[i,j]| \le \epsilon$ for all $i \neq j$, where $\frac{1}{2\sqrt{n}} \le \epsilon < \frac{1}{4}$.
	Then
	$$
	rank(D) \ge \frac{\mu \log n }{\epsilon^2 \log(1/\epsilon)}\;.
	$$
\end{lemma}

We also need the following lemma, which is the lower bound for separation dimension of $K_n$ proved by
Basavaraju et al. \cite{BCG+-2014}.
\begin{lemma}[\cite{BCG+-2014}]\label{l:pikn}
	$$
	\pi(K_n) \ge \log \lfloor n/2 \rfloor\;.
	$$
\end{lemma}

Let $\F$ be a separating family of $K_n$. For a pair of distinct vertices $u,v$ and for a vertex $w \notin \{u,v\}$, let $s_{\F}(u,v,w)$ be the number of elements $\sigma \in \F$ such that $w$ is before both $u,v$ in $\sigma$ or after both $u,v$ in $\sigma$. Let $s_{\F}(u,v)=\sum_{w \in [n] \setminus \{u,v\}} s_{\F}(u,v)$.

\begin{lemma}\label{l:avg}
	There exists a pair of distinct vertices  $u,v$ such that $s_{\F}(u,v) \ge \frac{2}{3}|\F|(n-2)$.
\end{lemma}
\begin{proof}
	Let $S$ be the sum of $s_{\F}(u,v)$ taken over all pairs of distinct vertices $u,v$.
	Each $\sigma \in \F$ contributes precisely $\frac{1}{3}n(n-1)(n-2)$ to $S$  implying that
	$S= \frac{1}{3}|\F|n(n-1)(n-2)$. By averaging, there is pair $u,v$ with
	$s_{\F}(u,v) \ge |S|/\binom{n}{2}=\frac{2}{3}|\F|(n-2)$.
\end{proof}

\vspace*{4mm}
\noindent {\em Proof of Theorem \ref{t:fair-lower}.}
Let $\F$ be a separating family of $K_n$ which is $\Delta$-balanced for $\Delta \le c \log n$ where $c$ is an absolute constant to be chosen later. Observe first that for any pair of disjoint edges $e,f$,
it holds that
\begin{equation}\label{e:fair}
\left| c_{\F}(\{e,f\}) -\frac{|\F|}{3} \right| \le c \log n\;.
\end{equation}
Indeed, suppose $e=xy$ and $f=zw$.
Then (\ref{e:fair}) holds since each element of $\F$ separates precisely one of the three pairs of edges $\{xy,zw\}$, $\{xz,yw\}$, $\{xw,yz\}$ and since $\F$ is $\Delta$-balanced.

Let $\frac{1}{48} > c > 0$ be a real number such that for all $\alpha \ge \frac{1}{2}$ it holds that
\begin{equation}\label{e:c}
\frac{\mu \alpha}{216 c^2 \log(\frac{\alpha}{6c})} < 1
\end{equation}
where $\mu$ is the constant from Lemma \ref{l:perturbed}.

We will prove that $|\F| \ge n/K$ where $K=3/(2\mu)$. Assume hereafter, in contradiction, that $|\F| \le n/K$.

Let $u,v$ be two vertices with $s_{\F}(u,v) \ge \frac{2}{3}|\F|(n-2)$ which exist by Lemma \ref{l:avg}.
By relabeling, we may assume that $\{u,v\}=\{n-1,n\}$. As in the proof of Lemma \ref{l:lower}, we consider the matrices $A,B,C=AA^T+BB^T$ defined there. Each diagonal entry $C[w,w]$ counts the number of $\sigma \in \F$ such that
$w$ is either before both $u,v$ or after both $u,v$. Hence, $tr(C)=s_{\F}(u,v)\ge \frac{2}{3}|\F|(n-2)$.
So, on average, a diagonal element has value $\frac{2}{3}|\F|$ while trivially, each diagonal entry does not exceed $|\F|$. Let $W \subseteq [n-2]$ be the set of vertices with $C[w,w] \ge |\F|/2$.
Then we have that $|W| \ge (n-2)/3$ as otherwise $tr(C) < \frac{2}{3}|\F|(n-2)$.
Let $C_W$ be the submatrix of $C$ obtained by restricting the rows and columns to $W$.
Then the dimension of $C_W$ is at least $(n-2)/3$ and clearly $rank(C_W) \le rank(C)$.
We will lower-bound $rank(C_W)$ and in fact show that $rank(C_W) > 2|\F|$ which contradicts the fact that
$rank(C) \le 2|\F|$ by (\ref{e:triv}).

Let $\alpha$ be such that $|\F|=\alpha \log n$.
Observe that by Lemma \ref{l:pikn} we have that $\alpha \ge \log \lfloor n/2 \rfloor/\log n$ (so $\alpha > \frac{1}{2}$)
and by our assumption, $\alpha \le n/(K \log n)$. We consider two cases, according to $\alpha$.
The first case is $6c\sqrt{n} \ge \alpha \ge \log \lfloor n/2 \rfloor/\log n$ and the second case is
$n/(K \log n) \ge \alpha > 6c\sqrt{n}$.

\vspace*{4mm}
\noindent {\em Case 1.} Here $6c\sqrt{n} \ge \alpha \ge \log \lfloor n/2 \rfloor/\log n$.
So, $C_W$ is a matrix of dimension $|W|$ whose diagonal entries are all at least $|\F|/2$ and by (\ref{e:fair}), the off-diagonal entries differ by at most $c\log n$ from $|\F|/3$. Let
$D = \frac{6}{|\F|}(C_W-\frac{|\F|}{3}J)$ where $J$ is the all-one matrix.
So the diagonal entries of $D$ are all at least $1$ and the off-diagonal entries are all, in absolute value, at most
$\frac{6}{|\F|}{c \log n} = 6c/\alpha$. Furthermore, we have
$$
\frac{1}{2\sqrt{|W|}} \le \frac{1}{2\sqrt{(n-2)/3}} \le \frac{1}{\sqrt{n}} \le \frac{6c}{\alpha} < \frac{1}{4}\;.
$$
By Lemma \ref{l:perturbed} and by the choice of $c$ satisfying (\ref{e:c}) we obtain that
\begin{align*}
rank(D) & \ge \frac{\mu \alpha^2 \log |W| }{36 c^2 \log(\alpha/(6c))}\\
&  \ge \frac{\mu \alpha^2 \log n }{72 c^2 \log(\alpha/(6c))}\\
& \ge  3\alpha \log n\\
& = 3|\F|\;.
\end{align*}
But since $rank(C) \ge rank(C_W) \ge rank(D)-1 \ge 3|\F|-1$ we have that $rank(C) > 2|\F|$, contradicting the fact
that $rank(C) \le 2|\F|$ by (\ref{e:triv}).

\vspace*{4mm}
\noindent {\em Case 2.} Here $n/(K \log n) \ge \alpha > 6c\sqrt{n}$.
This time we define $D = \frac{1}{\sqrt{n}c \log n}(C_W-\frac{|\F|}{3}J)$.
So the diagonal entries of $D$ are all at least $1$ and the off-diagonal entries are all, in absolute value, at most
$\frac{1}{\sqrt{n}} \ge \frac{1}{2\sqrt{|W|}}$. So, we can apply Lemma \ref{l:perturbed} with $\epsilon=1/\sqrt{n}$
and obtain that
$$
rank(D) \ge \frac{\mu n \log n }{\log (\sqrt{n})} = 2\mu n = 3n/K \ge 3|\F|\;.
$$
Concluding as in the previous case, we arrive at the exact same contradiction.
\qed

\section*{Acknowledgment}

I thank Abhiruk Lahiri for interesting discussions on separation dimension and the referees for helpful comments.

\bibliographystyle{plain}

\bibliography{references}

\end{document}